\documentclass[reqno, british]{amsart}

\usepackage{geometry}
\usepackage{mathrsfs}
\usepackage{enumitem}
\usepackage{amsthm}
\usepackage{amsmath}
\usepackage{amssymb}
\usepackage{hyperref}
\usepackage{graphicx, xypic}
\usepackage{float}

\makeatletter

\newtheorem*{thm*}{Theorem}
\theoremstyle{plain}
\newtheorem{thm}{\protect\theoremname}[section]

\theoremstyle{plain}
\newtheorem{prop}[thm]{\protect\propositionname}
\theoremstyle{definition}
\newtheorem{defn}[thm]{\protect\definitionname}
\theoremstyle{plain}
\newtheorem{lem}[thm]{\protect\lemmaname}
\theoremstyle{definition}
\newtheorem{example}[thm]{\protect\examplename}
\theoremstyle{remark}
\newtheorem{rem}[thm]{\protect\remarkname}
\ifx\proof\undefined
\newenvironment{proof}[1][\protect\proofname]{\par
\normalfont\topsep6\p@\@plus6\p@\relax
\trivlist
\itemindent\parindent
\item[\hskip\labelsep
\scshape
#1]\ignorespaces
}{%
\endtrivlist\@endpefalse
}
\providecommand{\proofname}{Proof}
\fi
\theoremstyle{plain}
\newtheorem{cor}[thm]{\protect\corollaryname}
\theoremstyle{plain}
\newtheorem{fact}[thm]{\protect\factname}
 \newlist{casenv}{enumerate}{4}
 \setlist[casenv]{leftmargin=*,align=left,widest={iiii}}
 \setlist[casenv,1]{label={{\itshape\ \casename} \arabic*.},ref=\arabic*}
 \setlist[casenv,2]{label={{\itshape\ \casename} \roman*.},ref=\roman*}
 \setlist[casenv,3]{label={{\itshape\ \casename\ \alph*.}},ref=\alph*}
 \setlist[casenv,4]{label={{\itshape\ \casename} \arabic*.},ref=\arabic*}

\makeatother

\usepackage{babel}
\providecommand{\definitionname}{Definition}
\providecommand{\examplename}{Example}
\providecommand{\lemmaname}{Lemma}
\providecommand{\remarkname}{Remark}
\providecommand{\theoremname}{Theorem}
\providecommand{\casename}{Case}
\providecommand{\corollaryname}{Corollary}
\providecommand{\factname}{Fact}
\providecommand{\propositionname}{Proposition}

\newcommand\Cref[1]{{Corollary~\ref{#1}}}
\newcommand\Lref[1]{{Lemma~\ref{#1}}}
\newcommand\Tref[1]{{Theorem~\ref{#1}}}

\newcommand\sSref[1]{{subsection~\ref{#1}}}

\newcommand\Dref[1]{{Definition~\ref{#1}}}

\global\long\def\R{\mathcal{R}}
\global\long\def\tr{\textnormal{tr}}

\global\long\def\layer#1#2{\overset{\left[#2\right]}{}#1}
\global\long\def\zero{\layer{0}{0}}

\global\long\def\one{\layer{0}{1}}
\global\long\def\minus{\layer{0}{-1}}
\global\long\def\zeroset#1{\overset{\left[0\right]}{}#1}
\global\long\def\ELTrop{\textnormal{ELTrop}}
\global\long\def\sgn{\textnormal{sgn}}
\global\long\def\adj{\textnormal{adj}}
\global\long\def\ELT#1#2{\mathscr{R}\left(#2,#1\right)}

\global\long\def\L{\mathscr{L}}
\global\long\def\F{\mathscr{F}}

\global\long\def\t{\tau}
\global\long\def\etr{\textnormal{etr}}

\global\long\def\Id{\textnormal{Id}}
\global\long\def\ghs{\underset{\text{gs}}{\vDash}}
\global\long\def\symmdash{\succeq_\circ}
\global\long\def\N{\mathbb{N}}
\global\long\def\quo#1#2{\raisebox{.2em}{\ensuremath{#1}}\left/\raisebox{-.2em}{\ensuremath{#2}}\right.}
\global\long\def\gr{\mathrm{gr}}

\newcommand{\RR}{\mathbb{R}}      

\begin{document}
\title{ELT Linear Algebra II}
\date{\today}

\author[Guy Blachar]{Guy Blachar}
\address{Department of Mathematics, Bar-Ilan University, Ramat-Gan 52900,
Israel.} \email{\href{mailto:blachag@biu.ac.il}{blachag@biu.ac.il}}

\author[Erez Sheiner]{Erez Sheiner}
\address{Department of Mathematics, Bar-Ilan University, Ramat-Gan 52900,
Israel.} \email{\href{mailto:erez@math.biu.ac.il}{erez@math.biu.ac.il}}

\thanks{This article contains work from Erez Sheiner's Ph.D.\ Thesis, which was accepted on 1.1.16, and from Guy Blachar's M.Sc.\ Thesis, both submitted to the Math Department at Bar-Ilan University. Both works were carried under the supervision of Prof.\ Louis Rowen from Bar-Ilan University, to whom we thank deeply for his help and guidance.}



\begin{abstract}
This paper is a continuation of \cite{BS}. Exploded layered tropical (ELT) algebra is an extension of tropical algebra with a structure of layers. These layers allow us to use classical algebraic results in order to easily prove analogous tropical results. Specifically we prove and use an ELT version of the transfer principal presented in \cite{Akian2008}.\\
In this paper we use the transfer principal to prove an ELT version of Cayley-Hamilton Theorem, and study the multiplicity of the ELT determinant, ELT adjoint matrices and quasi-invertible matrices.\\
We also define a new notion of trace -- the essential trace -- and study its properties.
\end{abstract}

\maketitle

\tableofcontents
\addtocontents{toc}{~\hfill\textbf{Page}\par}

\section{Introduction}

Tropical linear algebra, also known as Max-Plus linear algebra, has been studied for more than 50 years (ref.~\cite{B}). While tropical geometry mainly deals with geometric combinatorial problems, tropical linear algebra deals with algebraic non-linear combinatorial problems (for instance, the assignment problem \cite{K}). Tropical linear algebra may also be used as a mean to study the tropical algebraic geometry (for instance, the tropical resultant). Notable work in this field can be found at \cite{B}, \cite{DSS}, \cite{IR4}, \cite{IR3} and \cite{S}.\\

In our previous paper (\cite{BS}) we introduced a new structure, which we call exploded layered tropical algebra (or ELT algebra for short). This structure is a generalization of the work of Izhakian and Rowen (\cite{IR1}), and is similar to Parker's exploded structure (\cite{PR}). The layers enable us to use ``classical language'' even when dealing with tropical questions.\\

Our work in this paper can be divided into two main parts. The first one uses the theory of semirings with a negation map to study the ELT structure. We formulate and prove an ELT version of the transfer principles written in \cite{Akian2008}, and use them to study ELT matrix theory, such as the ELT adjoint matrix (\Tref{thm:Trans-Princ} and \Tref{thm:Trans-Princ-Surpass}).\\

The second part of our work deals with a new notion of trace. Whereas the trace can be defined as in the classical theory, it lacks some important properties in the ELT theory. For example, the trace of an ELT nilpotent matrix need not be of layer zero. We define the essential trace of an ELT matrix (\sSref{sec:etr}) to deal with such cases.\\

\subsection{ELT Algebras}

\begin{defn}
Let $\L$ be a semiring, and $\F$ a totally ordered semigroup. An \textbf{ELT algebra} is the pair $\R=\ELT{\F}{\L}$, whose elements are denoted $\layer a{\ell}$ for $a\in\F$ and $\ell\in\L$, together with the semiring (without zero) structure:
\begin{enumerate}
\item $\layer{a_{1}}{\ell_{1}}+\layer{a_{2}}{\ell_{2}}:=\begin{cases}
\layer{a_{1}}{\ell_{1}} & a_{1}>a_{2}\\
\layer{a_{2}}{\ell_{2}} & a_{1}<a_{2}\\
\layer{a_{1}}{\ell_{1}+_\L\ell_{2}} & a_{1}=a_{2}
\end{cases}$.
\item $\layer{a_{1}}{\ell_{1}}\cdot\layer{a_{2}}{\ell_{2}}:=\layer{\left(a_{1}+_\F a_{2}\right)}{\ell_{1}\cdot_\L\ell_{2}}$.
\end{enumerate}
We write . For $\layer{a}{\ell}$, $\ell$ is called the \textbf{layer}, whereas $a$ is called the \textbf{tangible value}.
\end{defn}

ELT algebras originate from \cite{PR}, and are also discussed in \cite{Sheiner2015}.\\

Let $\R$ be an ELT algebra. We write $s:\R\rightarrow\L$ for the projection on the first component (the \textbf{sorting map}):
$$s\left(\layer a{\ell}\right)=\ell$$
We also write $\t:\R\rightarrow\F$ for the projection on the second component:
$$\t\left(\layer a{\ell}\right)=a$$
We denote the \textbf{zero-layer subset}
$$\zeroset{\R}=\left\{\alpha\in \R\middle| s\left(\alpha\right)=0\right\}$$
and
$$\R^{\times}=\left\{\alpha\in\R\middle|s\left(\alpha\right)\neq 0\right\}=\R\setminus\zeroset{\R}$$\\

We note some special cases of ELT algebras.
\begin{example}
Let $\left(G,\cdot\right)$ be a totally ordered group. We denote by $G_{\max}$ the max-plus algebra defined over $G$, i.e.\ the set $G$ endowed with the operation
$$a\oplus b=\max\left\{a,b\right\},\;\;a\odot b=a\cdot b.$$
Then $G_{\max}$ is equivalent to the trivial ELT algebra with $\F=G$ and $\L=\left\{1\right\}$.
\end{example}

\begin{example}
Zur Izhakian's supertropical algebra (\cite{IZ}) is equivalent to an ELT algebra with a layering set $\L=\left\{1,\infty\right\}$, where
$$1+1=\infty,\;\;1+\infty=\infty+1=\infty,\;\;\infty+\infty=\infty$$
and
$$1\cdot 1=1,\;\;1\cdot\infty=\infty\cdot 1=\infty,\;\;\infty\cdot\infty=\infty.$$

The supertropical "ghost" elements $a^\nu$ correspond to $\layer{a}{\infty}$ in the ELT notation, whereas the tangible elements $a$ correspond to $\layer{a}{1}$.
\end{example}

We define a partial order relation $\vDash$ on $\R$ in the following way:
$$x\vDash y\Longleftrightarrow \exists z\in\zeroset{\R}:x=y+z$$

\begin{lem}[{\cite[Lemma 0.4]{BS}}]
$\vDash$ is a partial order relation on $\R$.
\end{lem}

Let us point out some important elements in any ELT algebra $\R$:
\begin{enumerate}
\item $\layer{0}{1}$, which is the multiplicative identity of $\R$.
\item $\layer{0}{0}$, which is idempotent for both operations of $\R$.
\item $\layer{0}{-1}$, which has the role of ``$-1$'' in our theory.
\end{enumerate}

Note that $\zero\cdot\layer{a}{\ell}=\layer{a}{0}$. Therefore, $\zeroset{\R}=\zero\,\R$. In particular, $\zeroset{\R}$ is an ideal of $\R$.\\

Throughout this paper, unless otherwise noted, we work under more general assumptions than in~\cite{BS}. Out underlying ELT algebras $\R=\ELT{\F}{\L}$ will be \textbf{commutative ELT rings}, meaning that $\F$ is an abelian group, and $\L$ is a commutative ring.

\subsection{The Element $-\infty$}\label{sub:the-element-minf}

As in the tropical algebra, ELT algebras lack an additive identity. Therefore, we adjoin a formal element to the ELT algebra $\R$, denoted by $-\infty$, which satisfies $\forall\alpha\in\R$:
$$\begin{array}{c}
-\infty+\alpha=\alpha+-\infty=\alpha\\
-\infty\cdot\alpha=\alpha\cdot-\infty=-\infty
\end{array}$$
We also define $s\left(-\infty\right)=0$. We denote $\overline{\R}=\R\cup\left\{-\infty\right\}$.\\

We note that $\overline{\R}$ is now a semiring, with the following property:
$$\alpha+\beta=-\infty\Longrightarrow\alpha=\beta=-\infty$$
Such a semiring is called an \textbf{antiring}. Antirings are dealt with in \cite{Tan2007} and \cite{Dolzan2008}.

\subsection{Non-Archimedean Valuations and Puiseux Series}

We recall the definition of a (non-Archimedean) valuation (see \cite{Efrat2006} and \cite{Tignol2015}).

\begin{defn}
Let $K$ be a field, and let $\left(\Gamma,+,\ge\right)$ be an abelian totally ordered group. Extend $\Gamma$ to $\Gamma\cup\left\{\infty\right\}$ with $\gamma<\infty$ and $\gamma+\infty=\infty+\gamma=\infty$ for all $\gamma\in\Gamma$. A function $v:K\to\Gamma\cup\left\{\infty\right\}$ is called a \textbf{valuation}, if the following properties hold:
\begin{enumerate}
  \item $v\left(x\right)=\infty\Longleftrightarrow x=0$.
  \item $\forall x,y\in K: v\left(xy\right)=v\left(x\right)+v\left(y\right)$.
  \item $\forall x,y\in K: v\left(x+y\right)\ge\min\left\{v\left(x\right), v\left(y\right)\right\}$.
\end{enumerate}
\end{defn}

Given a valuation $v$ over a field $K$, we recall some basic properties:
\begin{enumerate}
  \item $v\left(1\right)=0$.
  \item $\forall x\in K:v\left(-x\right)=v\left(x\right)$.
  \item $\forall x\in K^{\times}:v\left(x^{-1}\right)=-v\left(x\right)$.
  \item If $v\left(x+y\right)>\min\left\{v\left(x\right), v\left(y\right)\right\}$, then $v\left(x\right)=v\left(y\right)$. (For this reason, the equality between the valuation of two elements is central in out theory.)
\end{enumerate}

One may associate with $v$ the \textbf{valuation ring}
$$\mathcal{O}_v=\left\{x\in K\middle|v\left(x\right)\ge 0\right\}$$
This is a local ring with the unique maximal ideal
$$\mathfrak{m}_v=\left\{x\in K\middle|v\left(x\right)>0\right\}$$
The quotient $k_v=\quo{\mathcal{O}_v}{\mathfrak{m}_v}$ is called the \textbf{residue field} of the valuation.\\

Let us present another key construction related to valuations. For $\gamma\in\Gamma$, let $D_{\ge\gamma}=\left\{x\in K\middle|v\left(x\right)\ge\gamma\right\}$ and $D_{>\gamma}=\left\{x\in K\middle|v\left(x\right)>\gamma\right\}$. It is easily seen that $D_{\ge\gamma}$ is an abelian additive group, and that $D_{>\gamma}$ is a subgroup of $D_{\ge\gamma}$. Note that for $\gamma=0$, $D_{\ge 0}=\mathcal{O}_v$ and $D_{>0}=\mathfrak{m}_v$. Set $D_\gamma=\quo{D_{\ge\gamma}}{D_{>\gamma}}$. The \textbf{associated graded ring} of $K$ with respect to $v$ is
$$\gr_v\left(K\right)=\bigoplus_{\gamma\in\Gamma}D_\gamma$$
Given $\gamma_1,\gamma_2\in\Gamma$, the multiplication in $K$ induces a well-defined multiplication $D_{\gamma_1}\times D_{\gamma_2}\to D_{\gamma_1+\gamma_2}$ given by
$$\left(x_1+D_{>\gamma_1}\right)\cdot\left(x_2+D_{>\gamma_2}\right)=x_1x_2+D_{>\left(\gamma_1+\gamma_2\right)}$$
This multiplication can be extended to a multiplication map in $\gr_v\left(K\right)$, endowing it with a structure of a graded ring.\\

We will now focus on Puiseux series, which is the central example for our theory. The field of \textbf{Puiseux series} with coefficients in a field $K$ and exponents in an abelian ordered group $\Gamma$ is
$$K\{\{t\}\}=\left\{\sum_{i\in I}\alpha_i t^i\middle|\alpha_i\in K, I\subseteq \Gamma\text{ is well-ordered}\right\}$$
The resulting set, equipped with the natural operations, is a field; in addition, if $K$ is algebraically closed and $\Gamma$ is divisible, then $K\{\{t\}\}$ is also algebraically closed.\\

Assuming $\Gamma$ is also totally ordered, one may define a valuation on the field of Puiseux series $v:K\{\{t\}\}\to\Gamma\cup\left\{\infty\right\}$ as follows: $v\left(0\right)=\infty$, and
$$v\left(\sum_{i\in I}\alpha_i t^i\right)=\min\left\{i\in I\middle|\alpha_i\neq 0\right\}$$

Let us examine the associated graded ring with respect to this valuation. For each $\gamma\in\Gamma$, we first claim that $D_\gamma\cong K$. Indeed, the kernel of the homomorphism $f:D_{\ge\gamma}\to K$ defined by
$$f\left(\sum_{\gamma\leq i\in I}\alpha_i t^i\right)=\alpha_\gamma$$
is precisely $D_{>\gamma}$ (since $D_{>\gamma}$ is the subgroup of $D_{\ge\gamma}$ of Puiseux series whose minimal degree is bigger than $\gamma$).

\subsection{ELT Algebras and Puiseux Series}

Let $\R=\ELT{\F}{\L}$ be an ELT algebra. In \cite{BS} we introduced the \textbf{EL tropicalization} function $\ELTrop:\L\{\{t\}\}\rightarrow\overline{\R}$, which is defined in the following way: if $x\in \L\{\{t\}\}\backslash\left\{0\right\}$ has a leading monomial $\ell t^{a}$, then
$$\ELTrop\left(x\right)=\layer{\left(-a\right)}{\ell}.$$
In addition, $\ELTrop\left(0\right)=-\infty$.

\begin{lem}[{\cite[Lemma 0.9]{BS}}]\label{lem:ELTrop-prop}
The following properties hold:
\begin{enumerate}
\item $\forall x,y\in\L\{\{t\}\}:\ELTrop\left(x\right)+\ELTrop\left(y\right)\vDash\ELTrop\left(x+y\right)$.
\item $\forall\alpha\in\L\,\forall x\in\L\{\{t\}\}:\ELTrop\left(\alpha x\right)=\layer 0{\alpha}\, \ELTrop\left(x\right)$.
\item $\forall x,y\in\L\{\{t\}\}:\ELTrop\left(x\right)\ELTrop\left(y\right)\vDash\ELTrop\left(xy\right)$.
\end{enumerate}
\end{lem}
We remark that in the case in which $\R$ is an ELT integral domain, meaning $\L$ is an integral domain, we have $\ELTrop\left(x\right)\ELTrop\left(y\right)=\ELTrop\left(xy\right)$ for all $x,y\in\L\{\{t\}\}$. \\

Let us examine the relation $x\vDash\ELTrop\left(y\right)$ a bit more deeply. If $x=\ELTrop\left(y\right)$, it means that $x$ can be lifted to a Puiseux series which has $x$ as its leading monomial. Otherwise, we have that $x$ is of layer zero, and its tangible value is bigger than the tangible value of $\ELTrop\left(y\right)$; so one may say that $x$ can also be lifted to a Puisuex series with leading coefficient $x$, where we allow it to have a zero coefficient in its leading monomial.

\subsection{Semirings with a Negation Map and ELT Rings}\label{sub:ELT-symmetrized}

Semirings need not have additive inverses to all of the elements. While some of the theory of rings can be copied ``as-is'' to semirings, there are many facts about rings which use the additive inverses of the elements. The idea of negation maps on semirings (sometimes called symmetrized semirings) is to imitate the additive inverse map. Semirings with negation maps are discussed in \cite{Akian1990}, \cite{Gaubert1992}, \cite{Gaubert1997}, \cite{Akian2008}, \cite{Akian2014}, \cite{Rowen2016}.

\begin{defn}
Let $R$ be a semiring. A map $(-):R\to R$ is a \textbf{negation map} (or a \textbf{symmetry})
if the following properties hold:
\begin{enumerate}
\item $\forall a,b\in R:(-)\left(a+b\right)=(-)a+(-)b$.
\item $(-)0_R=0_R$.
\item $\forall a,b\in R:(-)\left(a\cdot b\right)=a\cdot\left((-)b\right)=\left((-)a\right)\cdot b$.
\item $\forall a\in R:(-)\left((-)a\right)=a$.
\end{enumerate}
We say that $\left(R,(-)\right)$ is a \textbf{semiring with a negation map} (or a \textbf{symmetrized semiring}). If $(-)$ is clear from the context, we will not mention it.
\end{defn}

We give several examples of semirings with negation maps:
\begin{itemize}
\item A trivial example of a negation map (over any semiring) is $(-)a=a$.
\item If $R$ is a ring, it has a negation map $(-)a=-a$.
\item If $\R$ is an ELT algebra, we have a negation map given by $(-)a=\minus a$.
\end{itemize}

The last example is the central example for our theory, since it shows that any ELT algebra is equipped with a natural negation map. Thus, the theory of semirings with negation maps can be used when dealing with ELT algebras.\\

We now present several notations from this theory:
\begin{itemize}
\item $a+(-)a$ is denoted $a^\circ$.
\item $R^\circ=\left\{a^\circ\middle|a\in R\right\}$.
\item We define two partial orders on $R$:
\begin{itemize}
\item The relation $\symmdash$ defined by
$$a\symmdash b\Leftrightarrow \exists c\in R^\circ:a=b+c$$
\item The relation $\nabla$ defined by
$$a\nabla b\Leftrightarrow a+(-)b\in R^\circ$$
\end{itemize}
\end{itemize}

If $\R$ is an ELT algebra, then some of these notations have already been defined. For example, $a^\circ=\zero a$, $\R^\circ=\zeroset{\R}$ and the relation $\symmdash$ is the relation $\vDash$.

\section{ELT Transfer Principle}

The transfer principles are two theorems, presented in \cite{Akian2008}, which allows to conveniently transfer equalities between polynomial expressions in the classical theory to theorems about semirings with a negation map. We recall that any ELT algebra $\R$ has a negation map
$$\left(-\right)\layer{a}{\ell}=\layer{a}{-\ell}$$
and thus we may view each ELT algebra as a semiring with a negation map.\\

In this section, we use the transfer principles to prove two transfer principles for the ELT theory, and use these transfer principles to study the ELT adjoint matrix.

\subsection{The Transfer Principle}

In this subsection, we briefly introduce the two classical transfer principles given in \cite{Akian2008}.

\begin{defn}
A \textbf{positive polynomial expression} in the variables $\lambda_1,\dots,\lambda_m$ is a formal expression produced by the context-free grammar $E\mapsto E+E,\left(E\right)\times\left(E\right)$, where the symbols $0,1,\lambda_1,\dots,\lambda_m$ are thought of as terminal symbols of the grammar. A \textbf{monomial} in a positive polynomial expression is a sum of expressions of the form $c_I\lambda_1^{i_1}\dots \lambda_m^{i_m}$, where $I=\left(i_1,\dots,i_n\right)$ is fixed.
\end{defn}

That means that $0,1,\lambda_1,\dots,\lambda_m$ are positive polynomial expressions. Also, $\lambda_1+\left(\lambda_2\right)\times\left(\lambda_1+\lambda_3\right)$ is a positive polynomial expression. Any positive polynomial expression $E$ can be interpreted as a polynomial in $\N\left[\lambda_1,\dots,\lambda_m\right]$. We say that a monomial $\lambda_1^{i_1}\dots \lambda_m^{i_m}$ \textbf{appears} in the expression $E$ if there exists a positive integer $c$ such that $c\lambda_1^{i_1}\dots \lambda_m^{i_m}$ appears in the expansion of the polynomial obtained by interpreting $E$ in $\N\left[\lambda_1,\dots,\lambda_m\right]$.

If $R$ is a semiring with a negation map, and if $c\lambda_1^{i_1}\dots \lambda_m^{i_m}$ is a monomial, we define
$$\left(-\right)\left(c\lambda_1^{i_1}\dots \lambda_m^{i_m}\right)=\left(\left(-\right)c\right)\lambda_1^{i_1}\dots \lambda_m^{i_m}$$

\begin{defn}
Let $R$ be a semiring with a negation map. A \textbf{polynomial expression} is a formal expression of the form $P^+-P^-$, where $P^+$ and $P^-$ are positive polynomial expressions. A \textbf{monomial} in $P$ is a sum of the monomials $c_I\lambda_1^{i_1}\dots \lambda_m^{i_m}$ from $P^+$ and $c'_I\lambda_1^{i_1}\dots \lambda_m^{i_m}$ from $P^-$, where $I=\left(i_1,\dots,i_n\right)$ is fixed. We say that a monomial \textbf{appears} in the polynomial expression $P$, if it appears either in $P^+$ or in $P^-$.
\end{defn}

\begin{defn}
Let $P$ and $Q$ be polynomial expressions. We say that the identity $P=Q$ is valid in a semiring with a negation map $R$, if it holds for any semiring with a negation map $R$ and for any substitution $\lambda_1=r_1,\dots,\lambda_m=r_m$ of $r_1,\dots,r_m\in R$.
\end{defn}

Recall the relations $\symmdash$ and $\nabla$ from \sSref{sub:ELT-symmetrized}.

\begin{defn}
Let $P$ and $Q$ be polynomial expressions.
\begin{enumerate}
\item We say that the identity $P\nabla Q$ holds in all commutative semirings with a negation map, if for any semiring with a negation map $R$ and for any substitution $\lambda_1=r_1,\dots,\lambda_m=r_m$ of $r_1,\dots,r_m\in R$,
    $$P\left(r_1,\dots,r_n\right)\nabla Q\left(r_1,\dots,r_n\right)$$
\item We say that the identity $P\symmdash Q$ holds in all commutative semirings with a negation map, if for any semiring with a negation map $R$ and for any substitution $\lambda_1=r_1,\dots,\lambda_m=r_m$ of $r_1,\dots,r_m\in R$,
    $$P\left(r_1,\dots,r_n\right)\symmdash Q\left(r_1,\dots,r_n\right)$$
\end{enumerate}
\end{defn}

We recall the transfer principle (\cite[Theorems 4.20 and 4.21]{Akian2008}):

\begin{thm}[Transfer principle, weak form]
Let $P$ and $Q$ be polynomial expressions. If the identity $P=Q$ holds in all commutative rings, then the identity $P\nabla Q$ holds in all commutative semirings with negation map.
\end{thm}

\begin{thm}[Transfer principle, strong form]
Let $P$ and $Q$ be polynomial expressions. If the identity $P=Q$ holds in all commutative rings, and if $Q=Q^+-Q^-$ for some positive polynomial expressions such that there is no monomial appearing simultaneously in $Q^+$ and $Q^-$, then the identity $P\symmdash Q$ holds in all commutative semirings with negation map.
\end{thm}

The transfer principle allows us to prove several important theorems in a rather convenient way.

\begin{cor}[Multiplicativity of determinant]\label{cor:trans-det-mult}
Let $R$ be a semiring with a negation map, and let $A\in R^{n\times n}$. We define:
\begin{eqnarray*}
\det\left(A\right)^+&=&\sum_{\sigma\in A_n}a_{1,\sigma\left(1\right)}\dots a_{n,\sigma\left(n\right)}\\
\det\left(A\right)^-&=&\sum_{\sigma\in S_n\setminus A_n}a_{1,\sigma\left(1\right)}\dots a_{n,\sigma\left(n\right)}\\
\det\left(A\right)&=&\det\left(A\right)^+-\det\left(A\right)^-
\end{eqnarray*}
In \cite[Corollary 4.8]{Akian2008}, it is proven that if $R$ is a semiring with a negation map, then
$$\forall A,B\in R^{n\times n}:\det\left(AB\right)\symmdash\det\left(A\right)\det\left(B\right)$$
\end{cor}

\begin{cor}[Cayley-Hamilton theorem]\label{cor:cayley-hamilton}
Let $R$ be a semiring with a negation map, and let $A\in R^{n\times n}$. We know that over a commutative ring, $f_A\left(A\right)=0$. We can use the strong form of the transfer principle componentwise, and thus
$$f_A\left(A\right)\symmdash 0$$
In other words,
$$f_A\left(A\right)\in \left(R^\circ\right)^{n\times n}$$
\end{cor}

\subsection{Formulation and Proof of the ELT Transfer Principle}

We first recall that all of our ELT algebras are commutative ELT rings, meaning $\R=\ELT{\F}{\L}$ where $\F$ is an abelian group and $\L$ is a commutative ring.\\

We would like to have a tool of proving polynomial identities over commutative ELT rings. Recall that any ELT ring is a semiring with a negation map $a\mapsto\layer{0}{-1}a$ (\sSref{sub:ELT-symmetrized}), and thus we may apply the transfer principle. In order to strengthen the general transfer principle, we will use results from tropical linear algebra. This is formulated in the following theorems:

\begin{thm}[ELT Transfer Principle for equality]\label{thm:Trans-Princ}
Let $P$ and $Q$ be polynomial expressions. Assume that the identity $P=Q$ holds in all commutative rings. If the identity $P=Q$ holds in all commutative tropical algebras, then the identity $P=Q$ holds in all commutative ELT rings.
\end{thm}

\begin{thm}[ELT Transfer Principle for surpassing]\label{thm:Trans-Princ-Surpass}
Let $P$ and $Q$ be polynomial expressions. Assume that the identity $P=Q$ holds in all commutative rings. If the identity $P\ge Q$ holds in all commutative tropical algebras, then the identity $P\vDash Q$ holds in all commutative ELT rings.
\end{thm}

We prove two lemmas which will help us prove the theorems:

\begin{lem}\label{lem:sum-zero-one-big-surpass}
Let $\R$ be an ELT algebra, and let $\alpha,\beta\in\R$. If $s\left(\alpha+\minus\beta\right)=0$, and if $\t\left(\alpha\right)\ge\t\left(\beta\right)$, then $\alpha\vDash\beta$.
\end{lem}
\begin{proof}
Denote $\alpha=\layer{a}{\ell}$, $\beta=\layer{b}{k}$. $\t\left(\alpha\right)\ge\t\left(\beta\right)$ means $a\ge b$. We have two options:
\begin{enumerate}
\item If $a=b$, then $\alpha+\minus\beta=\layer{a}{\ell-k}\in\zeroset{\R}$. Thus, $\ell=k$, which implies $\alpha=\beta$.
\item If $a>b$, then $\alpha+\minus\beta=\alpha\in\zeroset{\R}$. Thus, $\ell=0$, and $\alpha=\alpha+\beta\vDash\beta$.
\end{enumerate}
In any case $\alpha\vDash\beta$, and thus we are finished.
\end{proof}

\begin{lem}\label{lem:hom-elt-maxplus}
Let $\R=\ELT{\F}{\L}$ be an ELT algebra. Endow $\F$ with a max-plus algebra, $\F_{\max}$. Then the function $\t:\R\to\F_{\max}$ is a ``homomorphism'', in the sense that:
\begin{enumerate}
\item $\forall x,y\in\R:\t\left(x+y\right)=\t\left(x\right)\oplus\t\left(y\right)$.
\item $\forall x,y\in\R:\t\left(xy\right)=\t\left(x\right)\odot\t\left(y\right)$.
\end{enumerate}
\end{lem}
\begin{proof}
Take $\layer{a_{1}}{\ell_{1}},\layer{a_{2}}{\ell_{2}}\in\R$.
\begin{enumerate}
\item If $a_1=a_2$, then
$$\t\left(\layer{a_{1}}{\ell_{1}}+\layer{a_{2}}{\ell_{2}}\right)=\t\left(\layer{a_1}{\ell_1+_\L\ell_2}\right)=a_1=a_1\oplus a_2=\t\left(\layer{a_{1}}{\ell_{1}}\right)\oplus\t\left(\layer{a_{2}}{\ell_{2}}\right)$$
Otherwise, without loss of generality, $a_1>a_2$, and thus
$$\t\left(\layer{a_{1}}{\ell_{1}}+\layer{a_{2}}{\ell_{2}}\right)=\t\left(\layer{a_1}{\ell_1}\right)=a_1=a_1\oplus a_2=\t\left(\layer{a_{1}}{\ell_{1}}\right)\oplus\t\left(\layer{a_{2}}{\ell_{2}}\right)$$
\item $\t\left(\layer{a_{1}}{\ell_{1}}\cdot\layer{a_{2}}{\ell_{2}}\right)=\t\left(\layer{\left(a_{1}+_\F a_{2}\right)}{\ell_{1}\cdot_\L\ell_{2}}\right)=a_1+_\F a_2=a_1\odot a_2=\t\left(\layer{a_{1}}{\ell_{1}}\right)\odot\t\left(\layer{a_{2}}{\ell_{2}}\right)$
\end{enumerate}
\end{proof}

We now prove these theorems. We note that since \Tref{thm:Trans-Princ} follows from \Tref{thm:Trans-Princ-Surpass}, we will only prove the latter.\\

\begin{proof}[Proof of \Tref{thm:Trans-Princ-Surpass}]
By the weak form of the general transfer principle, $s\left(P+\minus Q\right)=0$.

Let $\R=\ELT{\F}{\L}$ be a commutative ELT ring. We will now prove that for any substitution $\lambda_1=r_1,\dots,\lambda_m=r_m$ of $r_1,\dots,r_m\in\R$,
$$\t\left(P\left(r_1,\dots,r_m\right)\right)\ge\t\left(Q\left(r_1,\dots,r_m\right)\right)$$
We endow $\F$ with the max-plus operations. By Lemma (\Lref{lem:hom-elt-maxplus}) $\t:\R\to\F_{\max}$ is a homomorphism. Therefore, for any ELT polynomial $p\in\R\left[\lambda\right]$ and for any substitution $\lambda_1=r_1,\dots,\lambda_m=r_m$ of $r_1,\dots,r_m\in\R$,
$$\t\left(p\left(r_1,\dots,r_m\right)\right)=p\left(\t\left(r_1\right),\dots,\t\left(r_m\right)\right)$$
Thus, for any substitution $\lambda_1=r_1,\dots,\lambda_m=r_m$ of $r_1,\dots,r_m\in\R$,
$$\t\left(P\left(r_1,\dots,r_m\right)\right)=P\left(\t\left(r_1\right),\dots,\t\left(r_m\right)\right)\ge Q\left(\t\left(r_1\right),\dots,\t\left(r_m\right)\right)=\t\left(Q\left(r_1,\dots,r_m\right)\right)$$

We have proven that $s\left(P+\minus Q\right)=0$ and that $\t\left(P\right)\ge\t\left(Q\right)$; by \Lref{lem:sum-zero-one-big-surpass}, we are finished.
\end{proof}

\begin{rem}
Throughout the uses of the ELT transfer principle for equality, we need to check the corresponding identity in commutative tropical algebras. However, since major work has been done in the supertropical theory (see \cite{IR3}, \cite{Izhaki2009} and \cite{Izhaki2010b}), we usually check that one of the following conditions holds:
\begin{enumerate}
\item The identity $P=Q$ holds in all commutative supertropical algebras.
\item The identity $\nu\left(P\right)=\nu\left(Q\right)$ holds in all commutative supertropical algebras.
\end{enumerate}
Similarly, to prove a surpassing relation, we usually check that one of the following conditions holds:
\begin{enumerate}
\item The identity $P\ghs Q$ holds in all commutative supertropical algebras.
\item The identity $P\ge_\nu Q$ holds in all commutative supertropical algebras.
\end{enumerate}
\end{rem}

This tool allows us to prove many polynomial surpassing and equalities without effort. An example is given in the next subsection.

\subsection{Multiplicity of the ELT Determinant}

We return to \Cref{cor:trans-det-mult}, which holds in particular over commutative ELT rings. We first formulate this corollary in the ``ELT language'':

\begin{cor}[Multiplicativity of the ELT determinant]\label{cor:elt-det-mult}
Let $\R$ be a commutative ELT ring. If~$A,B\in \left(\overline{\R}\right)^{n\times n}$, then
$$\det\left(AB\right)\vDash\det\left(A\right)\cdot\det\left(B\right)$$
\end{cor}

We now present several corollaries from the multiplicativity of the ELT determinant, which are two cases in which the ELT determinant is strictly multiplicative. First, we prove a lemma that will be helpful for the first case:

\begin{lem}\label{lem:non-zero-surpass-equal}
If $x,y\in\R$ satisfy $x\vDash y$, and if $s\left(x\right)\neq0$, then $x=y$.
\end{lem}
\begin{proof}
Write $x=\layer{a}{\ell}$, $y=\layer{b}{k}$. Since $x\vDash y$, there is $c\in\F$ such that $x=y+\layer{c}{0}$. In other words,
$$\layer{a}{\ell}=\layer{b}{k}+\layer{c}{0}$$
By the definition of addition, $a\ge b$, and $a=\max\left\{b,c\right\}$. If $a>b$, then $a=c>b$, and thus
$$\layer{a}{\ell}=\layer{b}{k}+\layer{c}{0}=\layer{c}{0}$$
in contradiction to the fact that $\ell=s\left(x\right)\neq 0$. Thus, $x=y$.
\end{proof}

\begin{cor}\label{cor:det-is-mult-non-zero-prod}
If $s\left(\det\left(AB\right)\right)\neq 0$, then
$$\det\left(AB\right)=\det\left(A\right)\cdot\det\left(B\right)$$
\end{cor}
\begin{proof}
From \Cref{cor:elt-det-mult},
$$\det\left(AB\right)\vDash\det\left(A\right)\cdot\det\left(B\right)$$
By \Lref{lem:non-zero-surpass-equal}, we get equality.
\end{proof}

Another case in which the determinant is multiplicative is when either $A$ or $B$ are invertible:

\begin{thm}\label{thm:det-is-mult-invert-mat}
If $A,B\in \left(\overline{\R}\right)^{n\times n}$, such that $A$ or $B$ are invertible. Then
$$\det\left(AB\right)=\det\left(A\right)\cdot\det\left(B\right)$$
\end{thm}
\begin{proof}
Assume that $B$ is invertible (the second direction is proved similarly). We note that by \Cref{cor:elt-det-mult},
$$\det\left(AB\right)\vDash\det\left(A\right)\cdot\det\left(B\right)$$
but also
$$\det\left(A\right)=\det\left(\left(AB\right)B^{-1}\right)\vDash\det\left(AB\right)\cdot\det \left(B^{-1}\right)=\det\left(AB\right)\cdot\left(\det\left(B\right)\right)^{-1}.$$
The latter surpassing implies that
$$\det\left(A\right)\cdot\det\left(B\right)\vDash\det\left(AB\right).$$
Since $\vDash$ is antisymmetric on $\overline{\R}$, the conclusion follows.
\end{proof}

Although the determinant is not multiplicative, a natural question is: if $AB$ is non-singular, is $BA$ also non-singular? The answer to this question is negative, as the following example demonstrates:
\begin{example}
In $\R=\ELT{\mathbb{R}}{\mathbb{C}}$, consider
$$A=\begin{pmatrix}\layer{1}{1} & \layer{1}{1}\\
\layer{2}{1} & \layer{3}{1}
\end{pmatrix}$$
Then $\det\left(AA^{t}\right)=\layer{8}{1}$, yet $\det\left(A^{t}A\right)=\layer{10}{0}$.
\end{example}

\subsection{The ELT Adjoint Matrix and Quasi-Invertible Matrices}
As we have seen, the invertible matrices in the ELT sense are limited. So, we shall try to generalize this. Our goal is to find an equivalent condition to the fact that $\det\left(A\right)$ is invertible.

\begin{defn}
Let $\R$ be a commutative ELT ring. A \textbf{quasi-identity matrix} is a matrix $\tilde{I}\in \left(\overline{\R}\right)^{n\times n}$, which is idemopotent, nonsingular and defined by
$$\left(\tilde{I}\right)_{i,j}=\begin{cases}
\one & i=j\\
\alpha_{i,j} & i\neq j
\end{cases}$$
where $\alpha_{i,j}\in\overline{\R}$, $s\left(\alpha_{i,j}\right)=0$.
\end{defn}

\begin{defn}
Let $\R$ be a commutative ELT ring, and let $A\in \left(\overline{\R}\right)^{n\times n}$. A matrix $B\in \left(\overline{\R}\right)^{n\times n}$ is a \textbf{quasi-inverse} for $A$, if $AB$ and $BA$ are quasi-identity matrices. In this case, $A$ is called \textbf{quasi-invertible}. Note that $AB$ and $BA$ may differ.
\end{defn}

\begin{defn}
Let $\R$ be a commutative ELT ring, and let $A\in \left(\overline{\R}\right)^{n\times n}$. The \textbf{$\left(i,j\right)$-minor} $A_{i,j}'$ of a matrix $A=\left(a_{i,j}\right)$ is obtained by deleting the $i$-th row and the $j$-th column. Its ELT determinant is $a_{i,j}'=\det A_{i,j}'$.
\end{defn}

\begin{defn}
Let $\R$ be a commutative ELT ring, and let $A\in \left(\overline{\R}\right)^{n\times n}$. The \textbf{adjoint} matrix of $A$ is
$$\left(\adj\left(A\right)\right)_{i,j}=\layer 0{\varepsilon\left(i,j\right)}\, a_{j,i}'$$
where $\varepsilon\left(i,j\right)=\left(-1\right)^{i+j}$.
\end{defn}

We would like to prove that when $\det\left(A\right)$ is invertible in $\R$, $\left(\det \left(A\right)\right)^{-1}\adj\left(A\right)$ is a quasi-inverse of~$A$. We present here some corollaries from the ELT transfer principle, which together will prove the assertion (\Tref{thm:det-inv-then-A-quasinv}). We use the ELT transfer principles componentwise.

\begin{cor}\label{cor:A-times-adj}
$A\cdot\adj\left(A\right)=\det\left(A\right)I_{A}$, where $I_{A}\vDash I$.
\end{cor}
\begin{proof}
We use the ELT transfer principle for surpassing. This result is known in ring theory, and is proved in the supertropical theory (see \cite[Remark 4.5]{IR3}).
\end{proof}

\begin{cor}\label{cor:det-of-A-times-adj}
$$\det\left(A\cdot\adj\left(A\right)\right)=\det\left(A\right)^n$$
\end{cor}
\begin{proof}
We use the ELT transfer principle for equalities. This result is known in ring theory, and is proved in the supertropical theory (see \cite[Theorem~4.9]{IR3}).
\end{proof}

\begin{cor}\label{cor:A-times-adj-squared}
$$\left(A\cdot\adj\left(A\right)\right)^2=\det\left(A\right)\cdot A\cdot\adj\left(A\right)$$
\end{cor}
\begin{proof}
We use the ELT transfer principle for equalities. This result is known in ring theory, and is proved in the supertropical theory (see \cite[Theorem~4.12]{IR3}).
\end{proof}

\begin{thm}\label{thm:det-inv-then-A-quasinv}
If $\det\left(A\right)$ is invertible in $\R$, then $A$ is quasi-invertible.
\end{thm}
\begin{proof}
By \Cref{cor:A-times-adj}, $A\cdot\adj\left(A\right)=\det\left(A\right)\cdot I_{A}$, where $I_{A}\vDash I$. It is left to prove that $I_A$ is nonsingular and idempotent.

By \Cref{cor:det-of-A-times-adj}, $\det\left(A\cdot\adj\left(A\right)\right)=\det\left(A\right)^n$. But
$$\det\left(A\right)^n=\det\left(A\cdot\adj\left(A\right)\right)=\det\left(\det\left(A\right)I_A\right)=\det\left(A\right)^n\det\left(I_A\right)$$
Since $\det\left(A\right)$ is invertible, $\det\left(I_A\right)=\one$.

To prove that $I_A$ is idempotent, we use \Cref{cor:A-times-adj-squared}:
\begin{eqnarray*}
I_A^2&=&\left(\left(\det\left(A\right)\right)^{-1}A\cdot\adj\left(A\right)\right)^2= \left(\det\left(A\right)^2\right)^{-1}\left(A\cdot\adj\left(A\right)\right)^2=\\ &=&\left(\det\left(A\right)^2\right)^{-1}\det\left(A\right)\cdot A\cdot\adj\left(A\right)=\left(\det\left(A\right)\right)^{-1}A\cdot\adj\left(A\right)=I_A
\end{eqnarray*}
as required.
\end{proof}

\begin{rem}
Using the same arguments, one may show that $\adj\left(A\right)\cdot A$ is $\det\left(A\right)$ times a quasi-identity matrix.
\end{rem}

\begin{cor}\label{cor:quasi-invertible-mat}
Let $\R$ be a commutative ELT ring. Then $A$ is quasi-invertible if and only if $\det\left(A\right)$ is invertible.
\end{cor}

We will now use the theory of the ELT adjoint matrix to study the connection between matrix singularity and linear dependency. We recall the following theorem:
\begin{thm*}[{\cite[Theorem 1.6]{BS}}]
Let $\R=\ELT{\RR}{\mathbb{F}}$ be an ELT algebra, where $\mathbb{F}$ is an algebraically closed field. Consider $A\in\left(\overline{\R}\right)^{n\times n}$. Then the rows of $A$ are linearly dependent, iff the columns of $A$ are linearly dependent, iff $s\left(\det A\right)=0_\mathbb{F}$.
\end{thm*}

This theorem was proved using the Fundamental Theorem, thus only in the case of $\R=\ELT{\RR}{\mathbb{F}}$ where $\mathbb{F}$ is an algebraically closed field. We will prove the following:

\begin{thm}\label{thm:det-is-invertible}
Let $\R$ be a commutative ELT ring, and let $A\in\left(\overline{\R}\right)^{n\times n}$ an ELT matrix. If $\det A$ is invertible in $\R$, then the columns (respectively, rows) of $A$ are linearly independent.
\end{thm}

Before proving this theorem, we recall the Hungarian algorithm (\cite{K}):
\begin{defn}
An entry $a_{i,j}$ of a tropical matrix $A\in\left(\overline{G_{\max}}\right)^{n\times n}$ is called \textbf{column-critical} if it is maximal within its columns, i.e., if $\forall k:a_{i,j}\ge a_{k,j}$. A matrix $A$ is called \textbf{critical} if there exists a permutation $\sigma\in S_n$ such that $a_{1,\sigma\left(1\right)},\dots,a_{n,\sigma\left(n\right)}$ are column-critical.
\end{defn}

\begin{thm}[The Hungarian Algorithm]
Let $A\in\left(\overline{G_{\max}}\right)^{n\times n}$ be a tropical matrix. Then there are scalars $\alpha_1,\dots,\alpha_n\in\R^{\times}$ such that
$$\left(\begin{matrix}&\alpha_1\odot R_1\left(A\right)&\\&\vdots&\\&\alpha_n\odot R_n\left(A\right)&\end{matrix}\right)$$
is critical. In other words, there exists a diagonal matrix $D\in\left(\overline{G_{\max}}\right)^{n\times n}$, whose diagonal entries are not $-\infty$, such that $A\odot D$ is critical.
\end{thm}

In the ELT case, we say that a matrix $A\in\left(\overline{\R}\right)^{n\times n}$ is \textbf{critical} if $\t\left(A\right)=\left(\t\left(a_{i,j}\right)\right)\in\left(\overline{G_{\max}}\right)^{n\times n}$ is critical.

\begin{cor}\label{cor:hung-alg-for-ELT}
Let $A\in\left(\overline{\R}\right)^{n\times n}$. Then there exists an invertible diagonal matrix $D\in\left(\overline{\R}\right)^{n\times n}$ such that $DA$ is critical.
\end{cor}
\begin{proof}
Let $D'=\left(d'_{i,j}\right)\in\left(\overline{G_{\max}}\right)^{n\times n}$ be a diagonal matrix such that $D'\odot\t\left(A\right)$ is a critical tropical matrix. We define $D=\left(d_{i,j}\right)\in\left(\overline{\R}\right)^{n\times n}$ as
$$d_{i,j}=\left\{\begin{matrix}\layer{d'_{i,i}}{1},&i=j\\-\infty,&i\neq j\end{matrix}\right.$$
Obviously, $DA$ is critical (since $\t\left(DA\right)=\t\left(D\right)\odot\t\left(A\right)=D'\odot\t\left(A\right)$ is critical), as required.
\end{proof}

We are now ready to prove \Tref{thm:det-is-invertible}.

\begin{proof}[Proof of \Tref{thm:det-is-invertible}]
We prove the assertion on the columns of $A$. The assertion that the rows of $A$ are linearly independent can be proven by replacing $A$ with $A^t$.\\

Suppose that $s\left(Av\right)=\left(0,\dots,0\right)^t$ for some $v\in\left(\overline{\R^{\times}}\right)^n$. If
$$A^{\nabla}=\left(\det A\right)^{-1}\adj\left(A\right),$$
there exists a quasi-identity matrix $I'_A$ such that
$$A^{\nabla}A=I'_A$$
Thus,
$$s\left(I'_Av\right)=s\left(A^{\nabla}A\cdot v\right)=\left(0,\dots,0\right)^t.$$

We apply \Cref{cor:hung-alg-for-ELT} for $\left(I'_A\right)^t$ to find a diagonal invertible matrix $D\in\left(\overline{\R}\right)^{n\times n}$ such that $D\left(I'_A\right)^t$ is critical. By transposing this matrix, we get a matrix $I'_AD$ such that there is a permutation $\sigma\in S_n$ for which $\left(I'_AD\right){i,j}$ is row-critical in $I'_AD$. By \Tref{thm:det-is-mult-invert-mat},
$$\det\left(I'_AD\right)=\det I'_A\cdot\det D=\det D$$
is not of layer zero, i.e.\ $I'_AD$ is nonsingular. We note that the only non-zero layered track in $I'_A$ is the diagonal track, since any entry of $I'_A$ which is not on the diagonal is of layer zero; hence, we may assume $\sigma=\Id$, that is the diagonal entries are row-critical.\\

Returning to the original equation $s\left(I'_Av\right)=\left(0,\dots,0\right)^t$, we replace $I'_A$ by $B=I'_AD$ and $v$ by~$v'=D^{-1}v$ to get
$$s\left(Bv'\right)=\left(0,\dots,0\right)^t$$
Let $1\leq k\leq n$ such that
$$\t\left(v'_k\right)=\max_{1\leq i\leq n}\t\left(v'_i\right).$$
Then
$$\left(Bv'\right)_k=\sum_{i=1}^n b_{k,i}v'_i=\sum_{\substack{i=1\\i\neq k}}^n b_{k,i}v'_i+b_{k,k}v'_k.$$
We recall that $B=I'_AD$, where $I'_A$ is a quasi-identity matrix and $D$ is an invertible diagonal matrix. Thus, every entry of $B$ which is not on its diagonal is of layer zero. Furthermore, for any $i\neq k$ we have $\t\left(v'_i\right)\leq\t\left(v'_k\right)$ and $\t\left(b_{k,i}\right)\leq\t\left(b_{k,k}\right)$. Therefore each summand $b_{k,i}v'_i$ cannot dominate~$b_{k,k}v'_k$, implying
$$\left(Bv'\right)_k=b_{k,k}v'_k$$
which implies $s\left(b_{k,k}v'_k\right)=0$. Since $b_{k,k}=\left(I'_A\right)_{k,k}d_{k,k}=d_{k,k}$ is not of layer zero (because $D$ is invertible), we must have $s\left(v'_k\right)=0$. Now, $v_k=d_{k,k}^{-1}v'_k$ implies $s\left(v_k\right)=0$.\\

But $v\in\left(\overline{\R^{\times}}\right)^n$, therefore $s\left(v_k\right)=0$ implies $v_k=-\infty$. By the choice of $k$, we must have $v=\left(-\infty,\dots,-\infty\right)$, which proves that the columns of $A$ are linearly independent.
\end{proof}

\subsection{The ELT Characteristic Polynomial and ELT Eigenvalues}
\begin{defn}
Let $A\in \left(\overline{\R}\right)^{n\times n}$ be a matrix. The \textbf{ELT characteristic polynomial} of $A$ is defined to be
$$p_A(\lambda)=\det\left(\lambda I_n+\layer{0}{-1}A\right)$$
\end{defn}

\begin{defn}
Let $A\in \left(\overline{\R}\right)^{n\times n}$ be a matrix. A vector $v\in (\overline{\R^{\times}})^n$ is called an \textbf{eigenvector} of~$A$ with an \textbf{eigenvalue} $x\in\overline{\R}$ if $v\neq (-\infty,...,-\infty)$ and
$$Av=xv.$$
\end{defn}

\begin{prop}
Let $A\in \left(\overline{\R}\right)^{n\times n}$ be a matrix with eigenvalue $x$. Then $s\left(p_A(x)\right)=0_{\L}$.
\end{prop}

\begin{proof}
Choose $v$ to be an eigenvector of the eigenvalue $v$, then $Av=xv$. Therefore
$$Av+\layer{0}{-1}\cdot xv = \layer{0}{0}\cdot xv.$$
In other words,
$$s\left(xv+\layer{0}{-1}Av\right)=(0_{\L},...,0_{\L}).$$
Thus
$$s\left((xI_n+\layer{0}{-1}A)v\right)=(0_{\L},...,0_{\L}).$$

By \Tref{thm:det-is-invertible}, we conclude that $s\left(\det\left(xI_n+\layer{0}{-1}A\right)\right)=0_{\L}$, i.e., $s\left(p_A\left(x\right)\right)=0_{\L}$.
\end{proof}

Note that the other direction is not necessarily true, i.e., there could be an ELT root of the characteristic polynomial which is not an eigenvalue. Indeed, if one tried to prove that direction, he would encounter the following difficulty:
$$Av \neq Av+\layer{0}{-1}\cdot xv + xv.$$

\begin{example}
Consider the matrix $A\in\left(\overline{\R}\right)^{2\times 2}$,
$$A=\begin{pmatrix}\layer{1}{1} & \layer{2}{1} \\ \layer{2}{1} & \layer{3}{1} \end{pmatrix},$$

Its ELT characteristic polynomial is
\begin{eqnarray*}
p_A(\lambda)&=&\det\left(\lambda I_2+\layer{0}{-1}A\right)=
\det\left(\begin{matrix}\lambda+\layer{1}{-1} & \layer{2}{-1} \\
\layer{2}{-1} & \lambda+\layer{3}{-1}
\end{matrix}\right)\\
&=&(\lambda+\layer{1}{-1})(\lambda+\layer{3}{-1}) + \layer{4}{-1}=\lambda^2 + \layer{3}{-1}\cdot\lambda + \layer{4}{0}.
\end{eqnarray*}
If $\lambda=\layer{1}{0}$, $\lambda=\layer{\alpha}{\ell}$ with $\alpha < 1$ or $\lambda=\layer{3}{1}$, then $s\left(p_A(\lambda)\right)=0$.\\

The only eigenvalue of $A$ is $\lambda=\layer{3}{1}$, with
$$\begin{pmatrix}\layer{1}{1} & \layer{2}{1} \\ \layer{2}{1} & \layer{3}{1} \end{pmatrix}
\begin{pmatrix}\layer{0}{1} \\ \layer{1}{1} \end{pmatrix} =
\begin{pmatrix}\layer{3}{1} \\ \layer{4}{1} \end{pmatrix}=\layer{3}{1}\begin{pmatrix}\layer{0}{1} \\ \layer{1}{1} \end{pmatrix}.$$\\
\end{example}

One may also define an ELT eigenvalue and an eigenvector in the following way:

\begin{defn}\label{ELT_eigen_vector}
Let $A\in \left(\overline{\R}\right)^{n\times n}$ be a matrix. A vector $v\in (\overline{\R^{\times}})^n$ is called an \textbf{ELT eigenvector} of~$A$ with an \textbf{ELT eigenvalue} $x\in\overline{\R}$ if $v\neq (-\infty,...,-\infty)$ and
$$s\left(Av+\layer{0}{-1}xv\right)=(0_{\L},...,0_{\L}).$$
\end{defn}

This definition is similar to the concept of 'ghost surpass' given by Izhakian, Knebusch and Rowen (ref. \cite{IKR}).\\

\begin{prop}
Let $\R=\ELT{\RR}{\mathbb{F}}$ be an ELT algebra, where $\mathbb{F}$ is an algebraically closed field, and let $A\in \left(\overline{\R}\right)^{n\times n}$ be a matrix. Then $x$ is an ELT eigenvalue of $A$ if and only if $s\left(p_A(x)\right)=0_{\L}$.
\end{prop}

\begin{proof}
By \Dref{ELT_eigen_vector}, $x$ is an eigenvalue of $A$ if and only if there exists a vector $v\in (\overline{\R^{\times}})^n$ such that $v\neq (-\infty,...,-\infty)$ and $$s\left(xv+\layer{0}{-1}Av\right)=(0_{\L},...,0_{\L}),$$
if and only if the columns of $p_A(x)=x\cdot I_n+\layer{0}{-1}A$ are linearly dependent, if and only if $p_A(x)$ is singular (\cite[Theorem 1.7]{BS}).
\end{proof}

We finish by reformulating Cayley-Hamilton theorem (\Cref{cor:cayley-hamilton}) in the ``ELT language''.

\begin{cor}[ELT Cayley-Hamilton theorem]\label{cor:elt-cayley-hamilton}
Let $\R$ be a commutative ELT ring, and let $A\in \left(\overline{\R}\right)^{n\times n}$ be an ELT matrix. Then
$$s\left(p_A\left(A\right)\right)=0$$
\end{cor}
\section{ELT Traces}
\subsection{Trace of ELT Matrices}
\begin{defn}
Let $\R$ be a commutative ELT ring, and take $A\in \left(\overline{\R}\right)^{n\times n}$, $A=\left(a_{i,j}\right)$. The \textbf{trace} of $A$ is
$$\tr\left(A\right)=\sum_{i=1}^{n}a_{i,i}$$
\end{defn}

\begin{lem}
The ELT trace satisfies the following relations:
\begin{enumerate}
\item $\forall A,B\in \left(\overline{\R}\right)^{n\times n}:\tr\left(A+B\right)=\tr\left(A\right)+\tr\left(B\right)$.
\item $\forall\alpha\in\overline{\R}\;\forall A\in \left(\overline{\R}\right)^{n\times n}:\tr\left(\alpha A\right)=\alpha\tr\left(A\right)$.
\item $\forall A,B\in \left(\overline{\R}\right)^{n\times n}:\tr\left(AB\right)=\tr\left(BA\right)$.
\end{enumerate}
\end{lem}
\begin{proof}
These properties can be proved just as the classical theory.
\end{proof}

\subsection{ELT Nilpotent Matrices}
\begin{defn}
Let $\R$ be a commutative ELT ring. A matrix $A\in \left(\overline{\R}\right)^{n\times n}$ is called \textbf{ELT nilpotent}, if there exists $m\in\N$ such that $A^m\in\zeroset{\left(\overline{\R}\right)^{n\times n}}$.
\end{defn}

Similarly to the classical theory, one would expect that the trace of a nilpotent matrix would be of layer zero; however, this is wrong. For this
reason, we define the essential trace in the next subsubsection.

\begin{example}
Let $\R=\ELT{\mathbb{R}}{\mathbb{C}}$, and consider the following matrix: $\displaystyle{A=\begin{pmatrix}\layer{0}{1} & \layer{1}{0}\\
\layer{0}{0} & \layer{0}{1}
\end{pmatrix}}$.
Then $\tr\left(A\right)=\layer{0}{2}$, while
$$A^{2}=\begin{pmatrix}\layer{0}{1} & \layer{1}{0}\\
\layer{0}{0} & \layer{0}{1}
\end{pmatrix}\begin{pmatrix}\layer{0}{1} & \layer{1}{0}\\
\layer{0}{0} & \layer{0}{1}
\end{pmatrix}=\begin{pmatrix}\layer{1}{0} & \layer{1}{0}\\
\layer{0}{0} & \layer{1}{0}
\end{pmatrix}.$$
Therefore, ELT nilpotent matrices don't necessarily  have zero-layered trace.
\end{example}

Another interesting example is an ELT nilpotent matrix, whose determinant is not of layer zero.
\begin{example}
Let $\R$ be a commutative ELT ring, and take $a,b,c,d\in\R$ such that
$$\t\left(a^2\right)<\t\left(bc\right)<\t\left(d^2\right)$$
and $s\left(d\right)=0$. Consider the matrix $\displaystyle{A=\begin{pmatrix}a&b\\c&d\end{pmatrix}}$. We have
$$A^2=\begin{pmatrix}a^2+bc&ab+bd\\ac+cd&bc+d^2\end{pmatrix}=\begin{pmatrix}bc&bd\\cd&d^2\end{pmatrix}$$
and
$$A^3=\begin{pmatrix}abc+bcd&abd+bd^2\\bc^2+cd^2&bcd+d^3\end{pmatrix}=\begin{pmatrix}bcd&bd^2\\cd^2&d^3\end{pmatrix}$$
which is of layer zero, since $s\left(d\right)=0$.

But we may choose $a,b,c,d$ such that $\t\left(bc\right)>\t\left(ad\right)$ and $s\left(bc\right)\neq0$; in that case, $A$ is quasi-invertible and ELT nilpotent.
\end{example}

\subsection{The Essential Trace}\label{sec:etr}

Before defining the new notion of trace, we give an important definition, which is significant in our construction of the new trace.
\begin{defn}
Let $\R$ be a commutative ELT ring, and let $\displaystyle{p\left(\lambda\right)=\sum_{i=1}^nh_i\in\overline{\R}\left[\lambda\right]}$ be an ELT polynomial, where each $h_i$ is a monomial. For a monomial $h$, define $\displaystyle{p_h\left(\lambda\right)=\sum_{h_i\neq h}h_i}$.
\begin{enumerate}
\item The monomial $h$ is called \textbf{inessential at a point $a\in\R$}, if $p\left(a\right)=p_h\left(a\right)$ and $\t\left(h\left(a\right)\right)<\t\left(p\left(a\right)\right)$.
    If $h$ is inessential at every point of $\R$, it is called \textbf{inessential}.
\item The monomial $h$ is called \textbf{essential at a point $a\in\R$}, if $p\left(a\right)=h\left(a\right)$ and $\t\left(p_h\left(a\right)\right)<\t\left(p\left(a\right)\right)$.
    If $h$ is essential at some point of $\R$, it is called \textbf{essential}.
\item The monomial $h$ is called \textbf{quasi-essential at a point $a\in\R$}, if it is neither inessential at $a$ nor essential at $a$.
    If $h$ is quasi-essential at some point of $\R$, it is called \textbf{quasi-essential}.
\end{enumerate}
\end{defn}

Throughout the rest of the section, we need a stronger assumption on our ELT algebras. We require them to be \textbf{divisible ELT fields}, that is ELT algebras of the form $\R=\ELT{\F}{\L}$, where $\F$ is a divisible group and $\L$ is a field.

\begin{defn}
Let $\R$ be a divisible ELT field, and let $A\in \left(\overline{\R}\right)^{n\times n}$. We define
\begin{eqnarray*}
L\left(A\right) & = & \left\{ \ell\ge1\middle|\forall k\leq n:\,\frac{\t\left(c_{\ell}\right)}{\ell}\ge\frac{\t\left(c_k\right)}{k}\right\} \\
\mu\left(A\right) & = & \min\left(L\left(A\right)\right)
\end{eqnarray*}
We write $\mu$ for $\mu\left(A\right)$, if $A$ is given.
\end{defn}

\begin{lem}\label{lem:etr-justification}
Let $\R$ be a divisible ELT field, and let $p\left(\lambda\right)=\lambda^{n}+{\displaystyle \sum_{i=1}^{n}\alpha_{i}\lambda^{n-i}}\in\overline{\R}\left[\lambda\right]$ be an ELT polynomial. Then the first monomial after $\lambda^{n}$ that is not inessential is $\alpha_{\mu}\lambda^{n-\mu}$.
\end{lem}
\begin{proof}
We need to find the monomial for which the intersection between $\lambda^{n}$ and $\alpha_{\ell}\lambda^{n-\ell}$ is maximal (in the sense that its tangible value is maximal).

First, we compute the tangible value of the intersection:
$$\lambda^{n}=\alpha_{\ell}\lambda^{n-\ell}\Rightarrow\t\left(\lambda\right)=\frac{\t\left(\alpha_\ell\right)}{\ell}$$
The tangible value of the value of the polynomial at that point is $\frac{n\t\left(\alpha_\ell\right)}{\ell}$.

So, if
$$\forall k\leq n:\,\frac{\t\left(\alpha_\ell\right)}{\ell}\ge\frac{\t\left(\alpha_k\right)}{k}$$
$\ell$ satisfies our conditions. Take such $\ell$ minimal, which is $\mu$, and we are done.
\end{proof}

\begin{rem}
If $\left|L\left(A\right)\right|\ge2$, then $\alpha_{\mu}\lambda^{n-\mu}$ is only quasi-essential in $p_A\left(\lambda\right)$.
\end{rem}

\begin{defn}
Let $\R$ be a divisible ELT field, and let $A\in \left(\overline{\R}\right)^{n\times n}$ be a matrix. Assume $p_{A}\left(\lambda\right)=\lambda^{n}+{\displaystyle \sum_{i=1}^{n}\alpha_{i}\lambda^{n-i}}$. $\alpha_{\mu}$ is called the \textbf{dominant characteristic coefficient}. The \textbf{essential trace} of $A$, denoted $\etr\left(A\right)$, is given by the formula
$$\etr\left(A\right)=\left\{\begin{matrix}\tr\left(A\right)&\minus\tr\left(A\right)\lambda^{n-1}\textnormal{ is essential in }p_A\left(\lambda\right)\\\layer{\left(\frac{\t\left(\alpha_\mu\right)}{\mu}\right)}{0}&\textnormal{Otherwise}\end{matrix}\right.$$
\end{defn}

\begin{lem}\label{lem:tr-zero-implies-etr-zero}
If $s\left(\tr\left(A\right)\right)=0$, then $s\left(\etr\left(A\right)\right)=0$.
\end{lem}
\begin{proof}
If $\minus \tr\left(A\right)\lambda^{n-1}$ is essential in $p_A\left(\lambda\right)$, then $\etr\left(A\right)=\tr\left(A\right)$.
which is of layer zero. Otherwise,~$\minus \tr\left(A\right)\lambda^{n-1}$ is not essential in $p_A\left(\lambda\right)$, and thus, by the definition of essential trace, $s\left(\etr\left(A\right)\right)=0$.
\end{proof}

\begin{defn}
Let $A\in \left(\overline{\R}\right)^{n\times n}$ be a matrix over a commutative ELT ring $\R$. A \textbf{path from $i$ to~$j$} is an expression of the form
$$s=a_{i,i_{2}}a_{i_{2},i_{3}}\dots a_{i_{k},j}$$
where $1\leq i_{1},\dots,i_{k+1}\leq n$. If $i=j$, we call it a \textbf{multicycle}. The \textbf{length} of $s$ is $\left|s\right|=k$. The \textbf{tangible average value} of $s$ is $\displaystyle{\frac{\t\left(s\right)}{\left|s\right|}}$. A \textbf{simple cycle} is a multicycle from $i$ to $i$, such that $i_{j}\neq i_{j'}$ for $j\neq j'$ (with $i_{1}=i$).
\end{defn}

\begin{fact}
Every multicycle can be written as a product of simple cycles.
\end{fact}

\begin{lem}\label{lem:i,j-ele-of-A^k}
The $\left(i,j\right)$ element in $A^{k}$ is the sum of all paths from $i$ to $j$. That is,
$$\left(A^{k}\right)_{i,j}=\sum_{1\leq i_{2},\dots,i_{k}\leq n}a_{i,i_{2}}a_{i_{2},i_{3}}\dots a_{i_{k},j}$$
\end{lem}
\begin{proof}
By induction on $k$, where the case $k=1$ is clear. If the assertion is true for some $k$, then
$$\left(A^{k+1}\right)_{i,j}=\sum_{\ell=1}^{n}a_{i,\ell}A_{\ell,j}^{k}=\sum_{j=1}^{n}\,\sum_{1\leq i_{2},\dots,i_{k}\leq n}a_{i,\ell}a_{\ell,i_{2}}a_{i_{2},i_{3}}\dots a_{i_{k},j}=\sum_{1\leq i_{2},\dots,i_{k+1}\leq n}a_{i,i_{2}}a_{i_{2},i_{3}}\dots a_{i_{k+1},j}$$
\end{proof}

\begin{lem}\label{lem:coeff-of-char-poly}
The coefficient of $\lambda^{n-k}$ in $p_{A}\left(\lambda\right)$ is
$$\left(\minus\right)^{k}\sum_{1\leq i_{1},\dots,i_{k}\leq n}\,\sum_{\sigma\in S_{k}}\layer 0{\sgn\sigma}\, a_{i_{1},i_{\sigma\left(1\right)}}a_{i_{2},i_{\sigma\left(2\right)}}\dots a_{i_{k},i_{\sigma\left(k\right)}}$$
\end{lem}
\begin{proof}
We must choose $n-k$ indices from which we ``take'' $\lambda$ in the expansion of $\det\left(\lambda I+\minus A\right)$; we are left with a $k\times k$ submatrix, with rows $i_{1},\dots,i_{k}$ from $A$. Its determinant is the inner sum.
\end{proof}

\begin{lem}\label{lem:Mulcyc-contr-to-etr}
Any multicycle contributing to the dominant characteristic coefficient must be a simple cycle.
\end{lem}
\begin{proof}
Otherwise, assume it is not a simple cycle. Since it can be written as a product of simple cycles, at least one of which, $s$, would have $\frac{\t\left(s\right)}{\left|s\right|}\ge\etr\left(A\right)$, and a shorter length. Thus, $s$ would give a dominant characteristic coefficient of lower degree, in contradiction to our assumption.
\end{proof}

\begin{lem}
If $s\left(\etr\left(A+B\right)\right)\neq0$, then $\etr\left(A+B\right)=\etr\left(A\right)+\etr\left(B\right)$.
\end{lem}
\begin{proof}
The assumption can only happen if $\minus\tr\left(A+B\right)\lambda^{n-1}$ is essential in $p_{A+B}\left(\lambda\right)$, and $\etr\left(A+B\right)=\tr\left(A+B\right)$.

The multicycles of $A+B$ are products of sums of elements from $A$ and from $B$. In particular, every multicycle of $A$ and of $B$ is a part of a multicycle of $A+B$.

We know that $\tr\left(A+B\right)=\tr\left(A\right)+\tr\left(B\right)$; if $\t\left(\tr\left(A+B\right)\right)=\t\left(\tr\left(A\right)\right)$, then $\t\left(\tr\left(A\right)\right)>$every average value of a multicycle of $A$; so $\etr\left(A\right)=\tr\left(A\right)$. We are almost finished:
\begin{casenv}
\item If $\t\left(\tr\left(A+B\right)\right)=\t\left(\tr\left(A\right)\right)=\t\left(\tr\left(B\right)\right)$, then $\etr\left(A\right)=\tr\left(A\right)$, $\etr\left(B\right)=\tr\left(B\right)$, and $$\etr\left(A+B\right)=\tr\left(A+B\right)=\tr\left(A\right)+\tr\left(B\right)=\etr\left(A\right)+\etr\left(B\right)$$

\item If $\t\left(\tr\left(A+B\right)\right)=\t\left(\tr\left(A\right)\right)>\t\left(\tr\left(B\right)\right)$, then $\etr\left(A\right)=\tr\left(A\right)$, and $$\etr\left(A+B\right)=\tr\left(A+B\right)=\tr\left(A\right)+\tr\left(B\right)=\tr\left(A\right)=\etr\left(A\right)$$
    In particular, $\t\left(\etr\left(A\right)\right)>\t\left(\etr\left(B\right)\right)$; otherwise, there would have been a multicycle from $B$ with $\t\left(\textnormal{average value}\right)\ge\t\left(\etr\left(A\right)\right)=\t\left(\etr\left(A+B\right)\right)$, and thus $s\left(\etr\left(A+B\right)\right)=0$, which is a contradiction. So
    $$\etr\left(A+B\right)=\etr\left(A\right)=\etr\left(A\right)+\etr\left(B\right)$$
\end{casenv}
\end{proof}

\begin{example}
In general, it is not true that $\etr\left(A+B\right)\vDash\etr\left(A\right)+\etr\left(B\right)$. For example, take in $\left(\ELT{\mathbb{R}}{\mathbb{C}}\right)^{2\times 2}$ the following matrices: $A=\left(\begin{matrix}\layer{0}{1}&\layer{0}{1}\\-\infty&\layer{0}{1}\end{matrix}\right)$, $B=A^t$. Then $\etr\left(A\right)=\etr\left(B\right)=\layer{0}{2}$. However, $A+B=\left(\begin{matrix}\layer{0}{2}&\layer{0}{1}\\\layer{0}{1}&\layer{0}{2}\end{matrix}\right)$, and $p_{A+B}\left(\lambda\right)=\lambda^2+\layer{0}{-4}\lambda+\layer{0}{3}$. The monomial $\layer{0}{-4}\lambda$ is quasi-essential, and thus $\etr\left(A+B\right)=\layer{0}{0}$.
\end{example}

\begin{lem}
$\etr\left(AB\right)=\etr\left(BA\right)$.
\end{lem}
\begin{proof}
By \Lref{lem:Mulcyc-contr-to-etr}, it is enough to check only simple cycles. Assume
$$\left(AB\right)_{i_{1},i_{2}}\left(AB\right)_{i_{2},i_{3}}\dots\left(AB\right)_{i_{k},i_{1}}$$
contributes to the dominant characteristic coefficient, where $i_{j}\neq i_{j'}$ for $j\neq j'$. So there are $\ell_{1},\dots,\ell_{k}$ such that
$$s=a_{i_{1},\ell_{1}}b_{\ell_{1},i_{2}}a_{i_{2},\ell_{2}}b_{\ell_{2},i_{3}}\dots a_{i_{k},\ell_{k}}b_{\ell_{k},i_{1}}$$
contributes to the dominant characteristic coefficient, i.e. $\t\left(s\right)=\etr\left(AB\right)$.

If $\ell_{j}=\ell_{j'}$ for $j<j'$, write
\begin{eqnarray*}
s & = & \left(a_{i_{1},\ell_{1}}b_{\ell_{1},i_{2}}\dots a_{i_{j},\ell_{j}}b_{\ell_{j},i_{j'+1}}\dots a_{i_{k},\ell_{k}}b_{\ell_{k},i_{1}}\right)\left(b_{\ell_{j},i_{j+1}}a_{i_{j+1},\ell_{j+1}}\dots a_{i_{j'},\ell_{j'}}\right)=\\
 & = & \underbrace{\left(a_{i_{1},\ell_{1}}b_{\ell_{1},i_{2}}\dots a_{i_{j},\ell_{j}}b_{\ell_{j},i_{j'+1}}\dots a_{i_{k},\ell_{k}}b_{\ell_{k},i_{1}}\right)}_{s_{1}}\underbrace{\left(a_{i_{j+1},\ell_{j+1}}\dots a_{i_{j'},\ell_{j'}}b_{\ell_{j},i_{j+1}}\right)}_{s_{2}}
\end{eqnarray*}
Since $\t\left(s\right)=\t\left(\etr\left(AB\right)\right)$, $\t\left(s_{1}\right)\ge\t\left(\etr\left(AB\right)\right)$ or $\t\left(s_{2}\right)\ge\t\left(\etr\left(AB\right)\right)$. Since $s_{1}$ and $s_{2}$ are shorter, we get a contradiction.

So $\ell_{j}\neq\ell_{j'}$ for $j\neq j'$. So
$$s=b_{\ell_{1},i_{2}}a_{i_{2},\ell_{2}}b_{\ell_{2},i_{3}}\dots a_{i_{k},\ell_{k}}b_{\ell_{k},i_{1}}a_{i_{1},\ell_{1}}$$
is a part of the simple cycle
$$\left(BA\right)_{\ell_{1},\ell_{2}}\left(BA\right)_{\ell_{2},\ell_{3}}\dots\left(BA\right)_{\ell_{k},\ell_{1}}$$
in $BA$. So every simple cycle contributing to $\etr\left(AB\right)$ also contributes to $\etr\left(BA\right)$.

By symmetry, the opposite is true as well; so $\etr\left(AB\right)=\etr\left(BA\right)$.
\end{proof}

\begin{lem}\label{lem:multicycle-ge-tr}
If there is a multicycle $s=a_{i_1,i_{2}}a_{i_{2},i_{3}}\dots a_{i_{k},i_1}$ of length $k\ge 2$ such that $$\t\left(\tr\left(A\right)^k\right)<\t\left(s\right)$$
then $\minus\tr\left(A\right)\lambda^{n-1}$ is not essential in $p_{A}\left(\lambda\right)$ (meaning, it is either inessential or quasi-essential). In other words, $\mu\left(A\right)\ge 2$.
\end{lem}
\begin{proof}
Firstly, we may assume that $k\leq n$; otherwise, write $s$ as a product of multicycle, $s=s_1\dots s_j$. Since
$$\t\left(\tr\left(A\right)^k\right)<\t\left(s\right)=\sum_{\ell=1}^j\t\left(s_\ell\right)$$
At least one $s_\ell$ should satisfy $\t\left(\tr\left(A\right)^{\left|s_\ell\right|}\right)<\t\left(s_\ell\right)$, and we may replace $k$ by $\ell$.

Therefore, we assume $k\leq n$. Write $p_{A}\left(\lambda\right)=\lambda^{n}+{\displaystyle \sum_{i=1}^{n}}\alpha_{i}\lambda^{n-i}$. By \Lref{lem:coeff-of-char-poly},
$$c_k=a_{i_1,i_{2}}a_{i_{2},i_{3}}\dots a_{i_{k},i_1}+\cdots=s+\cdots$$
In particular, $\t\left(c_k\right)\ge\t\left(s\right)$. Also recall that $c_1=\minus\tr\left(A\right)$. Therefore,
$$\t\left(c_1\right)=\t\left(\tr\left(A\right)\right)=\frac{\t\left(\tr\left(A\right)^k\right)}{k}< \frac{\t\left(s\right)}{k}\leq\frac{\t\left(c_k\right)}{k}$$
An thus $\mu\left(A\right)\ge k\ge 2$, as required.
\end{proof}

\begin{lem}\label{lem:ELT-nilpotent-etr}
If $A$ is ELT nilpotent, and if $s\left(\tr\left(A\right)\right)\neq0$, then $\minus\tr\left(A\right)\lambda^{n-1}$ is not essential in $p_{A}\left(\lambda\right)$. In other words, $\minus\tr\left(A\right)\lambda^{n-1}$ is either quasi-essential or inessential in $p_{A}\left(\lambda\right)$.
\end{lem}
\begin{proof}
For this proof, write $a\leq_{\t}b$ if $\t\left(a\right)\leq\t\left(b\right)$, $a<_{\t}b$ if $\t\left(a\right)<\t\left(b\right)$, and $a\equiv_{\t}b$ if $\t\left(a\right)=\t\left(b\right)$.

We take some $a_{i,i}$ such that $a_{i,i}\equiv_{\t}\tr\left(A\right)$ and $s\left(a_{i,i}\right)\ne0$; without loss of generality, let $i=1$. Write $p_{A}\left(\lambda\right)=\lambda^{n}+{\displaystyle \sum_{i=1}^{n}}\alpha_{i}\lambda^{n-i}$.

By \Lref{lem:i,j-ele-of-A^k}, $\left(A^{k}\right)_{1,1}=a_{1,1}^{k}+\cdots$. Take $k$ minimal such that $s\left(A^{k}\right)=0$. In particular, $s\left(\left(A^{k}\right)_{1,1}\right)=0$, so we have two cases:

\begin{casenv}
\item There is a multicycle $s=a_{1,i_{2}}a_{i_{2},i_{3}}\dots a_{i_{k},1}$ with $a_{1,1}^{k}<_{\t}s$. Firstly, we may assume that $k\leq n$; otherwise, write $s$ as a product of multicycle, $s=s_1\dots s_j$. Since
    $$\t\left(\tr\left(A\right)^k\right)<\t\left(s\right)=\sum_{\ell=1}^j\t\left(s_\ell\right)$$
    At least one $s_\ell$ should satisfy $\t\left(\tr\left(A\right)^{\left|s_\ell\right|}\right)<\t\left(s_\ell\right)$, and we may replace $k$ by $\ell$.

    Therefore, we assume $k\leq n$. Write $p_{A}\left(\lambda\right)=\lambda^{n}+{\displaystyle \sum_{i=1}^{n}}\alpha_{i}\lambda^{n-i}$. By \Lref{lem:coeff-of-char-poly},
    $$c_k=a_{i_1,i_{2}}a_{i_{2},i_{3}}\dots a_{i_{k},i_1}+\cdots=s+\cdots$$
    In particular, $\t\left(c_k\right)\ge\t\left(s\right)$. Also recall that $c_1=\minus\tr\left(A\right)$. Therefore,
    $$\t\left(c_1\right)=\t\left(\tr\left(A\right)\right)=\frac{\t\left(\tr\left(A\right)^k\right)}{k}< \frac{\t\left(s\right)}{k}\leq\frac{\t\left(c_k\right)}{k}$$
    An thus $\mu\left(A\right)\ge k\ge 2$, as required.

\item There is no multicycle with $a_{1,1}^{k}<_{\t}a_{1,i_{2}}a_{i_{2},i_{3}}\dots a_{i_{k},1}$. Then there has to be a multicycle for which $a_{1,1}^{k}\equiv_{\t}a_{1,i_{2}}a_{i_{2},i_{3}}\dots a_{i_{k},1}=s$ (since $s\left(\left(A^{k}\right)_{1,1}\right)=0$, and $a_{1,1}^k$ is a summand in the sum defining $\left(A^{k}\right)_{1,1}$).

    By writing $s$ as a product of simple cycles, we may assume that each simple cycle $a_{j_{1},j_{2}}a_{j_{2},j_{3}}\dots a_{j_{\ell},j_{1}}\equiv_{\t}a_{1,1}^{\ell}$ (Otherwise, we are finished by the first case).

    Since $\ell<n$, we get that $\left|L\left(A\right)\right|\ge2$, meaning $\minus\tr\left(A\right)\lambda^{n-1}$ is quasi-essential in $p_{A}\left(\lambda\right)$.
\end{casenv}
\end{proof}

\begin{cor}\label{cor:etr-of-nilpotent}
If $A$ is ELT nilpotent, then $s\left(\etr\left(A\right)\right)=0$.
\end{cor}
\begin{proof}
There are two cases:
\begin{enumerate}
\item If $s\left(\tr\left(A\right)\right)=0$, then $s\left(\etr\left(A\right)\right)=0$ by \Lref{lem:tr-zero-implies-etr-zero}.
\item Otherwise, $s\left(\tr\left(A\right)\right)\neq 0$; but then $s\left(\etr\left(A\right)\right)=0$ by \Lref{lem:ELT-nilpotent-etr}.
\end{enumerate}
\end{proof}

\bibliographystyle{plain}
\bibliography{references}

\end{document}